\documentclass[12pt]{article}
\usepackage[utf8]{inputenc}
\usepackage[T1]{fontenc}

\usepackage{amsmath, amsthm}
\usepackage{amssymb}
\usepackage{hyperref}

\newtheorem{theorem}{Theorem}[section]
\newtheorem{proposition}[theorem]{Proposition}
\newtheorem{corollary}[theorem]{Corollary}
\newtheorem{lemma}[theorem]{Lemma}

\DeclareMathOperator{\lps}{lps}
\DeclareMathOperator{\lpps}{lpps}
\DeclareMathOperator{\lpp}{lpp}
\DeclareMathOperator{\lppp}{lppp}
\DeclareMathOperator{\trim}{trim}
\DeclareMathOperator{\mirror}{mirror}

\title{Upper Bound for Palindromic and Factor Complexity of Rich Words}

\author{Josef Rukavicka\thanks{Department of Mathematics,
Faculty of Nuclear Sciences and Physical Engineering, Czech Technical University in Prague
(josef.rukavicka@seznam.cz).}}

\newtheorem{definition}[theorem]{Definition}
\theoremstyle{remark}
\newtheorem{example}[theorem]{Example}
\newtheorem{remark}[theorem]{Remark}

\date{\small{October 21, 2019}\\
   \small Mathematics Subject Classification: 68R15}

\begin{document}
\maketitle

\begin{abstract}
A finite word $w$ of length $n$ contains at most $n+1$ distinct palindromic factors. If the bound  $n+1$ is attained, the word $w$ is called rich. An infinite word $w$ is called rich if every finite factor of $w$ is rich.

Let $w$ be a word (finite or infinite) over an alphabet with $q>1$ letters, let $F(w,n)$ be the set of factors of length $n$ of the word $w$, and let $F_p(w,n)\subseteq F(w,n)$ be the set of palindromic factors of length $n$ of the word $w$. 

We present several upper bounds for $\vert F(w,n)\vert$  and $\vert F_p(w,n)\vert$, where $w$ is a rich word. In particular we show that \[\vert F(w,n)\vert \leq (q+1)8n^2(8q^{10}n)^{\log_2{2n}}+q\mbox{.}\]

In 2007, Bal{\'a}{\v z}i, Mas{\'a}kov{\'a}, and Pelantov{\'a} showed that \[\vert F_p(w,n)\vert +\vert F_p(w,n+1)\vert \leq \vert F(w,n+1)\vert-\vert F(w,n)\vert+2\mbox{,}\] where $w$ is an  infinite word whose set of factors is closed under reversal. We generalize this inequality for finite words.
\end{abstract}

\section{Introduction}
The field of combinatorics on words includes the study of \emph{palindromes} and \emph{rich} words. In recent years there have appeared several articles concerning this topic \cite{Ba_phd,BlBrLaVu11,DrJuPi,ScSh16}. Recall that a palindrome is a word that is equal to its \emph{reversal}, such as ``noon'' and ``level''. A word is called rich if it contains the maximal number of palindromic factors. It is known that a word of length $n$ can contain at most $n+1$ palindromic factors, including the empty word \cite{DrJuPi}. An infinite word $w$ is rich if every finite factor of $w$ is rich.

Rich words possess various properties; see, for instance, \cite{BaPeSta2,BuLuGlZa2,GlJuWiZa}. We will use two of them. The first uses the notion of a \emph{complete return}. Given a word $w$ and a factor $r$ of $w$, we call the factor $r$ a \emph{complete return} to $u$ in $w$ if $r$ contains exactly two occurrences of $u$, one as a prefix and one as a suffix. A property of rich words is that all complete returns to any palindromic factor $u$ in $w$ are palindromes \cite{GlJuWiZa}. 

The second property of rich words that we use says that a factor $r$ of a rich word $w$ is uniquely determined by its longest palindromic prefix and its longest palindromic suffix \cite{BuLuGlZa2}. Some generalizations of this property may be found in \cite{10.1007/978-3-319-66396-8_7}.

In the current article we present upper bounds for the palindromic and factor complexity of rich words. In other words, this means that we derive upper bounds for the number of palindromes and factors of given length in a rich word $w$. There are already some related results; see below.

Let us define $F(w,n)$ to be the set of factors of length $n$ of $w$, let $F_p(w,n)$ be the set of palindromic factors of length $n$ of $w$, and let $F(w)=\bigcup_{j\geq 0}F(w,j)$, where $w$ is a finite or infinite word. Let $w^R$ denote the reversal of $w=w_1w_2\cdots w_{n-1}w_n$, where $w_i$ are letters; formally $w^R=w_nw_{n-1}\cdots w_2w_1$. We say that a set $S$ of finite words is \emph{closed under reversal} if $w\in S$ implies that $w^R\in S$. 

It is clear that $\vert F_p(w,n)\vert \leq \vert F(w,n)\vert$. Some less obvious inequalities are known. One of the interesting inequalities is the following one \cite{BaMaPe, BaPeSta2}.
Given an infinite word $w$ with $F(w)$ closed under reversal, then \begin{equation}\label{rju589w65f5n2m6a}\vert F_p(w,n)\vert +\vert F_p(w,n+1)\vert \leq \vert F(w,n+1)\vert-\vert F(w,n)\vert+2\mbox{.}\end{equation} In order to prove the inequality (\ref{rju589w65f5n2m6a}) the authors used the notion of Rauzy graphs; a Rauzy graph is a subgraph of the de Bruijn graph \cite{SALIMOV2010_rauzy}. In Section $3$ we generalize this result for every word (finite or infinite) with $F(w,n+1)$ closed under reversal, which allows us to improve our upper bound from Section $2$ for the factor complexity of finite rich words. 

In \cite{AlBaCaDa}, another inequality has been proven for infinite non-ultimately periodic words: $\vert F_p(w,n)\vert< \frac{16}{n}\vert F(w,n+\lfloor \frac{n}{4}\rfloor)\vert$.

In \cite{RuSh16}, the authors show that a random word of length $n$ contains, on expectation, $\Theta(\sqrt{n})$ distinct palindromic factors.

Let $\Pi(n)$ denote the number of rich words of length $n$. If $w$ is a rich word then obviously $\vert F(w,n)\vert \leq \Pi(n)$. Hence the number of rich words forms the upper bound for the palindromic and factor complexity of rich words. 
The number of rich words was investigated in \cite{Vesti2014}, where the author gives a recursive lower bound on the number of rich words of length $n$, and an upper bound on the number of binary rich words. Better results can be found in \cite{GuShSh15}.
The authors of \cite{GuShSh15} construct for each $n$  a large set of rich words of length $n$. Their construction gives,  currently,  the best lower bound on the number of binary rich words, namely
$\Pi(n)\geq \frac{C^{\sqrt{n}}}{p(n)}$,
 where $p(n)$ is a polynomial and the constant $C \approx 37$. \\
Every factor of a rich word is also rich \cite{GlJuWiZa}. In other words,  the language of rich words is factorial. In particular, this means that
$\Pi(n)\Pi(m)\geq \Pi(n+m)$ for all $m, n\in \mathbb{N}$. Therefore,  Fekete's lemma  implies the existence of the limit of  $\sqrt[n]{\Pi(n)}$, and moreover \cite{GuShSh15}:
\[
\lim\limits_{n\rightarrow \infty}\sqrt[n]{\Pi(n)}= \inf \left \{\sqrt[n]{\Pi(n)} \colon n \in \mathbb{N} \right\}.
\]
For a fixed $n_0$,   one can find the number of all rich words of length $n_0$ and obtain an upper bound on the limit.
Using a computer Rubinchik counted $\Pi(n)$ for $n\leq 60$; see the sequence \url{https://oeis.org/A216264}. 
As  $ \sqrt[60]{\Pi(60)} < 1.605$, he obtained an upper bound for the binary alphabet: $\Pi(n)<c 1.605^n$ for some constant $c$ \cite{GuShSh15}.

In \cite{RukavickaRichWords2017}, the author shows that $\Pi(n)$ has a subexponential
growth on every finite alphabet.  Formally $\lim\limits_{n\rightarrow \infty}\sqrt[n]{\Pi(n)}=1$.
This result  is  an argument in favor of a conjecture formulated in \cite{GuShSh15}  saying that  for some infinitely growing function $g(n)$ the following holds for a binary alphabet:  
\[
{\Pi(n)} = \mathcal{O} \Bigl(\frac{n}{g(n)}\Bigr)^{\sqrt{n}}\mbox{.}
\]

As already mentioned, we construct upper bounds for palindromic and factor complexity of rich words. The proof uses the following idea. Let $u$ be a palindromic factor of a rich word $w$ on the alphabet $A$, such that $aub$ is factor of $w$, where $a,b\in A$ and $a\not =b$. Let $\lpp(w)$ and $\lps(w)$ denote the longest palindromic prefix and suffix of $w$ respectively. Then $\lpp(aub)$ and $\lps(aub)$ uniquely determine the factor $aub$ in $w$ \cite{BuLuGlZa2}. Let $\lpps(w)$ denote the longest proper palindromic suffix of $w$. We show that $a,b$ and $\lpps(u)$ also uniquely determine $aub$. In addition, we observe that either $\vert \lpps(u)\vert\leq \frac{1}{2}\vert u\vert$ or $u$ contains a palindromic factor $\bar u$ that uniquely determines $u$ and $\vert \bar u\vert \leq \frac{1}{2}\vert u\vert$. We obtain a ``short'' palindrome and letters $a,b$ that uniquely determine the ``long'' palindrome $u$ in the case when $aub$ is a factor of $w$. In these ``short'' palindromes there are again other ``shorter'' palindromes, and so on. As a consequence we present an upper bound for the number of factors of the form $aub$ with $\vert aub \vert=n$.
The property of rich words that all complete returns to any palindromic factor $u$ in $w$ are palindromes \cite{GlJuWiZa} allows us to prove that if $w$ contains the factors $xux$ and $yuy$, where $x,y \in A$ and $x\not =y$, then $w$ must contain a factor of the form $aub$, where $a,b\in A$ and $a\not =b$. This property demonstrates the relation between the factors $aub$ and palindromic factors $xux$. Due to this we derive an upper bound for the palindromic complexity of rich words. With the upper bound for palindromic complexity, the property that each factor is uniquely determined by its longest palindromic prefix and suffix \cite{BuLuGlZa2}, and the inequality (\ref{rju589w65f5n2m6a}) we obtain several upper bounds on palindromic and factor complexity. Principally we show that $\vert F_p(w,n)\vert, \vert F(w,n)\vert \leq n^{c\log_2{n}}$, where $n>2$ and $c$ is some constant depending on the size of the alphabet.

\section{Palindromic and factor complexity of rich words}
Consider an alphabet $A$ with $q$ letters, where $q>1$. Let $A^+=\bigcup_{j>0}A^j$ denote the set of all nonempty words over $A$, where $A^j$ is the set of words of length $j$.

Let $\epsilon$ denote the empty word, let $A^*=A^+\cup \{\epsilon\}$, and let $$A^{\infty}=\{w_1w_2w_3\cdots \mid w_i\in A\mbox{ and }i>0\}$$ be the set of infinite words.

Let $R_n\subseteq A^n$ be the set of rich words of length $n\geq 0$. Let $R^+=\bigcup_{j>0}R_j$ and $R^*=R^+\cup\{\epsilon\}$. In addition, we define $R^{\infty}\subseteq A^{\infty}$ to be the set of infinite rich words. Let $R=R^+\cup R^{\infty}$.

Let $\lps(w)$ and $\lpp(w)$ be the longest palindromic suffix and the longest palindromic prefix of a word $w\in A^*$ respectively. Additionally, we introduce $\lpps(w)$ to be the longest proper palindromic suffix and $\lppp(w)$ to be the longest proper palindromic prefix, where $\vert w\vert>1$; proper means that $\lpps(w)\not=w$ and $\lppp(w)\not =w$. For a word $w$ with $\vert w\vert\leq 1$ we define $\lppp(w)=\lpps(w)=\epsilon$.

Let $w=w_1w_2\cdots w_n$ be a word, where $w_i\in A$. We define $w[i]=w_i$ and $w[i,j]= w_iw_{i+1}\cdots w_j$, where $0<i\leq j\leq n$.

Moreover we define the following notation:
\begin{itemize}
\item
$P_n\subset A^n$: the set of palindromes of length $n\geq 0$.
\item
$P^+=\bigcup_{j> 0}P_j$ (the set of all nonempty palindromes).
\item
$F(w)$: the set of factors of the word $w\in A^*\cup A^{\infty}$.
\item
$F(w,n) = \{u\mid u\in F(w) \mbox{ and }\vert u\vert=n\}$ (the set of factors of length $n$).
\item
$F_p(w)=(P^+\cup\{\epsilon\})\cap F(w)$ (the set of palindromic factors).
\item
$F_p(w,n) = F(w,n)\cap P_n$ (the set of palindromic factors of length $n$).
\end{itemize}

\begin{definition}
Let $\trim(w)=w[2,\vert w\vert -1]$, where  $w\in A^*$ and $\vert w\vert> 2$. For $\vert w\vert\leq2$ we define $\trim(w)=\epsilon$. If $S$ is a set of words, then 
\[
\trim(S)=\{\trim(v)\mid v\in S\}\mbox{.}
\]
\end{definition}
\begin{remark}
The function $\trim(w)$ removes the first and last letter from $w$. 
\end{remark}
\begin{example} Suppose that  $A=\{0,1,2,3,4,5\}$. 
\begin{itemize}
\item
$\trim(01123501)=112350$.
\item
$\trim(\{12213, 112, 2,344\})=\{221,1,\epsilon,4\}$.
\end{itemize}
\end{example}
We will deal a lot with the words of the form $aub$, where $u$ is a palindrome and $a,b$ are distinct letters. Hence we introduce some more notation for them.
\begin{definition}
Given $w\in R$ and $n>2$, let
\[
\begin{split}
\gamma(w,n)=\{aub\mid aub\in F(w,n)\mbox{ and }u\in F_p(w,n-2) \\ \mbox{ and }a,b\in A\mbox{ and }a\not=b\}\mbox{.}
\end{split}
\]
If $n\leq 2$ then we define $\gamma(w,0)=\gamma(w,1)=\gamma(w,2)=\emptyset$.

Let $\bar \gamma(w,n) =\bigcup_{aub\in \gamma(w,n)}\{(u,a), (u,b)\}$, where $a,b\in A$.
Let $aub\in \gamma(w,n)$, where $a,b\in A$. We call the word $aub$ a $u$-\emph{switch} of $w$. Alternatively we say that $w$ contains a $u$-\emph{switch}.
\end{definition}
\begin{remark}
Note that a pair $(u,a)\in \bar \gamma(w,n)$ if and only if there exists $b\in A$ such that $aub\in \gamma(w,n)$ or $bua\in \gamma(w,n)$.
\end{remark}
\begin{example}
If $A=\{0,1,2,3,4,5,6\}$ and \[w=5112211311001131133114111146\] then:
\begin{itemize}
\item
$\gamma(w,8)=\{51122113,31133114,14111146\}$.
\item
$\trim(\gamma(w,8))=\{112211,113311,411114\}$.
\item
$\bar \gamma(w,8)=\{(112211,3),(112211,5),(113311,3), (113311,4), $\\
$(411114,1), (411114,6)\}$.
\item
$w$ does not contain $110011$-switch. Formally $110011\not \in \trim(\gamma(w,8))$.
\end{itemize}
\end{example}

\begin{remark}
The idea of a $u$-switch is inspired by the next lemma. If a rich word $w$ contains palindromes $aua$, $bub$, where $a,b\in A$, $a\not=b$, and $\vert aua\vert=\vert bub\vert=n$, then $w$ contains a $u$-switch of length $n$. The $u$-switch ``switches'' from $a$ to $b$. Note that $aua, bub\in F(w)$ does not imply that $aub\in F(w)$ or $bua \in F(w)$. It may be, for example, that $auc, cub\in F(w)$.  Nonetheless $(u,a),(u,b)\in \bar \gamma(w,n)$. The next lemma shows that if $aua, xuy\in F(w,n)$ then $(u,a)\in \bar \gamma(w,n)$, where $x,y\in A$ and $a\not =x$ or $a\not =y$.
\end{remark}

\begin{lemma}
\label{uk5855d}
Suppose $w\in R$ and suppose $u\in F_p(w,n-2)$, where $n> 2$. If $a,b_1,b_2\in A$, $\vert \{a,b_1,b_2\}\vert>1$, and $aua, b_1ub_2\in F(w,n)$ then $(u,a)\in \bar \gamma(w,n)$.
\end{lemma}
\begin{remark}
The condition $\vert \{a,b_1,b_2\}\vert>1$ in Lemma \ref{uk5855d} means that at least one letter is different from the others.
\end{remark}
\begin{proof}
Let $r$ be a factor of $w$ such that $aua$ is unioccurrent in $r$ and $\trim(r)$ is a complete return to $u$ in $w$. Since $aua$ and $b_1ub_2$ are factors of $w$, it is obvious that such $r$ exists. Clearly there are $x_1,x_2,y_1,y_2\in A$ such that $x_1ux_2$ is a prefix of $r$ and $y_1uy_2$ is a suffix of $r$. The complete return $\trim(r)$ to $u$ is a palindrome \cite{GlJuWiZa}. Hence $x_2=y_1$. Since $aua$ is unioccurrent in $r$, it follows that $x_2=y_1=a$, $x_1\not=y_2$, and $a\in \{x_1,y_2\}$. In consequence we have that $(u,a)\in \bar \gamma(w,n)$.
\end{proof}
To clarify the previous proof, let us consider the following two examples. For both examples suppose that $A=\{1,2,3,4,5,6\}$.
\begin{example}
Let $w=321234321252126$.
Let $aua=32123$ and $b_1ub_2=52126$. Then $r=32123432125$ and $\trim(r)=212343212$ is a complete return to $212$. Therefore $(212,3)\in \bar \gamma(w,5)$. Note that $b_1ub_2$ is not a factor of $r$.
\end{example}
\begin{example}
Let $w=321234321252$.
Let $aua=32123$ and $xuy=b_1ub_2=32125$. Then $r=32123432125$ and $\trim(r)=212343212$ is a complete return to $212$. Therefore $(212,3)\in \bar \gamma(w,5)$. Note that $b_1ub_2$ is a factor of $r$.
\end{example}
We show that the number of palindromic factors and the number of $u$-switches are related.
\begin{proposition}
\label{gt569wp}
If $w\in R$ and $n>2$ then
\[2\vert \gamma(w,n) \vert + \vert F_p(w,n-2) \vert \geq \vert F_p(w,n)\vert\mbox{.}\]
\end{proposition}
\begin{proof}
Let $\omega(w,n)=\{aua\vert (u,a)\in \bar \gamma(w,n)\}$. Less formally said, $\omega(w,n)$ is a set of palindromes of length $n$ such that if $w$ contains a $u$-switch $aub$ then $aua,bub\in \omega(w,n)$. Obviously we have that \begin{equation}\label{gkkdj5d4f85}\vert \omega(w,n)\vert \leq 2\vert \gamma(w,n)\vert\mbox{.}\end{equation}

Let \[\tilde F_p(w,n)=\{v\mid v\in F_p(w,n)\mbox{ and }\trim(v)\in\trim(\gamma(w,n))\}\]
and \[\dot F_p(w,n)=\{v\mid v\in F_p(w,n)\mbox{ and }\trim(v)\not\in\trim(\gamma(w,n))\}\mbox{.}\]

Obviously $F_p(w,n)=\tilde F_p(w,n)\cup \dot F_p(w,n)$ and $\tilde F_p(w,n)\cap \dot F_p(w,n)=\emptyset$. It follows that 
\begin{equation}
\label{eq5s6d8984}
\vert \tilde F_p(w,n)\vert + \vert \dot F_p(w,n)\vert=\vert F_p(w,n)\vert\mbox{.}
\end{equation}
Suppose $v\in F_p(w,n)$ and let $u=\trim(v)$.
\begin{itemize}
\item
If $v \in \tilde F_p(w,n)$ then $w$ contains a $u$-switch. From Lemma  \ref{uk5855d} it follows that $v\in \omega(w,n)$; this and (\ref{gkkdj5d4f85}) imply that  
\begin{equation}\label{gg5587d4}\vert \tilde F_p(w,n)\vert \leq \vert \omega(w,n)\vert \leq 2\vert \gamma(w,n)\vert\mbox{.}\end{equation}
\item
If $v \not \in \tilde F_p(w,n)$ then $w$ does not contain a $u$-switch. We have that $u\in F_p(w,n-2)\setminus \trim(\gamma(w,n))$. Obviously if $t\in F_p(w,n-2)\setminus \trim(\gamma(w,n))$, $a,b\in A$, and $w$ has palindromic factors $ata$ and $btb$, then $a=b$ since $w$ does not contain a $t$-switch. It follows that 
\begin{equation}\label{ttg5f4f58r}\vert \dot F_p(w,n)\vert\leq \vert F_p(w,n-2)\vert\mbox{.}\end{equation}
\end{itemize}
The proposition follows from (\ref{eq5s6d8984}), (\ref{gg5587d4}), and (\ref{ttg5f4f58r}).
\end{proof}
To clarify the previous proof, let us consider the following example.
\begin{example}
If $A=\{0,1,2,3,4,5,6,7,8\}$ and
\[w=2110112333211011454110116110116778776\] then
\begin{itemize}
\item
$\gamma(w,7)=\{2110114,4110116\}$,
\item
$F_p(w,7)=\{1233321,2110112,1145411,6110116,6778776\}$,
\item
$\tilde F_p(w,7)=\{2110112, 6110116\}$,
\item
$\dot F_p(w,7)=\{1233321, 1145411,6778776\}$,
\item
$F_p(w,5)=\{23332,11011,14541,77877\}$,
\item
$2\vert \gamma(w,7) \vert + \vert F_p(w,5) \vert \geq \vert F_p(w,7)\vert$, and
\item
$4 + 4>5$.
\end{itemize}
\end{example}
In the next proposition we show that if $a,b$ are different letters and $aub$ is a switch of a rich word $w$ then the longest proper palindromic suffix $r$ of $u$ and the letters $a,b$ uniquely determine the palindromic factor $u\in F_p(w)$. 

\begin{proposition}
\label{tu9856rg65k}
If $w\in R$, $u,v\in F_p(w)$, $\lpps(u)=\lpps(v)$, $a,b\in A$, $a\not = b$, and $aub,avb\in F(w)$ then $u=v$.
\end{proposition}
\begin{proof}
It is known that if $r,t$ are two factors of a rich word $w$ and $\lps(r)=\lps(t)$ and $\lpp(r)=\lpp(t)$, then $r=t$ \cite{BuLuGlZa2}. We will identify a $u$-switch by the longest proper palindromic suffix of $u$ and two distinct letters $a,b$ instead of by the functions $\lps$ and $\lpp$.

Given a $u$-switch $aub$ where $a\not =b$, $a,b\in A$, we know that $\lps(aub)$ and $\lpp(aub)$ uniquely determine the factor $aub$ in $w$.
We will prove that for given $a,b\in A$, $a\not =b$, $n\geq 0$, and a palindrome $r$ there is at most one palindrome $u\in F_p(w)$ such that $\lpps(u)=r$ and $aub\in \gamma(w, \vert aub\vert)$. 

Suppose, to get a contradiction, that there are $u,v\in F_p(w)$, $u\not =v$, $a,b\in A$, $a\not =b$ such that $\lps(aub)=bpb$, $\lps(avb)=bsb$, $\lpp(aub)=axa$, $\lpp(avb)=aya$, $\lpps(u)=\lpps(v)=r$, and $aub,avb\in \bigcup_{j>0}\gamma(w,j)$. This implies that $p,s,x,y$ are prefixes of $r$. Thus if $x\not =y$, then $\vert x\vert\not=\vert y\vert$. Without loss of generality, let $\vert x\vert<\vert y\vert$. Since $y$ is a prefix of $r$, either $ya$ is a prefix of $r$ or $r=y$. Consequently $aya$ is a prefix of both $aub$ and $avb$, and this contradicts the assumption that $\lpp(aub)=axa$; $aya$ is a prefix of $aub$ and $\vert aya\vert>\vert axa\vert$. Analogously if $p\not =s$. It follows that  $x=y$ and $p=s$. Therefore $\lpp(aub)=\lpp(avb)$ and $\lps(aub)=\lps(avb)$, which would imply that $u=v$, which is a contradiction.

Hence we conclude that $a,b\in A$, $a\not =b$, and a palindrome $r$ determine at most one palindrome $u\in F_p(w)$ such that $\lpps(u)=r$ and $u\in \trim(\gamma(w,\vert u\vert+2))$.
\end{proof}
In the following we derive an upper bound for the number of $u$-switches. We need one more definition to be able to partition the set $\trim(\gamma(w,n+2))$ into subsets based on the longest proper palindromic suffix.

\begin{definition}
Given $w\in R$, $r\in R^+$ and $n\geq 0$, let
\[\Upsilon(w,n,r)=\{u\mid u\in \trim(\gamma(w,n+2)) \mbox{ and }\lpps(u)=r\}\mbox{.}\]
\end{definition}
\begin{remark}
The set $\Upsilon(w,n,r)$ contains palindromic factors $u$ of length $n$ of the word $w$ having the longest proper palindromic suffix equal to $r$ and such that $w$ contains a $u$-switch.
Obviously $\bigcup_{r\in F_p(w)}\Upsilon(w,n,r)=\trim(\gamma(w,n+2))$ and $\Upsilon(w,n,r)\cap\Upsilon(w,n,\bar r)=\emptyset$ if $r\not=\bar r$.
\end{remark}

\begin{example}
If $A=\{0,1,2,3,4,5\}$ and
$w=5112211311001131133114$ then
\begin{itemize}
\item
$\gamma(w,6)=\{51122113, 31133114\}$,
\item
$\Upsilon(w,6,11)=\{112211,113311\}$, and 
\item
$110011\not \in \Upsilon(w,6,11)$, because $w$ does not contain a $110011$-switch.
\end{itemize}
\end{example}
A simple corollary of the previous proposition is that the size of the set $ \Upsilon(w,n,r)$ is limited by the constant $q(q-1)$. Recall that $q$ is the size of the alphabet $A$. 
\begin{corollary}
\label{heu586wlp8}
If $w,r\in R$ and $n\geq 0$ then $\vert \Upsilon(w,n,r) \vert\leq q(q-1)\mbox{.}$
\end{corollary}
\begin{proof}
From Proposition \ref{tu9856rg65k} it follows that  $$\vert \Upsilon(w,n,r)\vert \leq \vert\{(a,b)\mid a,b\in A\mbox{ and }a\not=b\}\vert=q(q-1)\mbox{.}$$ In other words, $\vert \Upsilon(w,n,r)\vert$ is equal or smaller that the number of pairs of distinct letters $(a,b)$.
\end{proof}
We define $\bar \Gamma(w,n)=\max\{\vert \gamma(w,i)\vert\mid 0\leq i\leq n \}$, where $w\in R$ and $n\geq 0$. Furthermore we define $\Gamma(w,n)=\max\{1,\bar \Gamma(w,n)\}$.
\begin{remark}
We defined $\Gamma(w,n)$ as the maximum from the set of sizes of $\gamma(w,i)$, where $0\leq i\leq n$. In addition, we defined that $\Gamma(w,n)\geq 1$. This is just for practical reason. In this way we can find a constant $c$ such that $\Gamma(w,n_1)=c\Gamma(w,n_2)$ for any $n_1,n_2\geq 0$. The function $\Gamma(w,n)$ will allow us to present another relation between the number of palindromic factors of length $n$ and the number of $u$-switches without using $F_p(w,n-2)$.
\end{remark}
\begin{lemma}
\label{pfw88q} If $w\in R$ and $n>0$ then
$$(q+1)n\Gamma(w,n)\geq \vert F_p(w,n)\vert\mbox{.}$$
\end{lemma} 
\begin{proof}
We define two functions $\bar \phi$ and $\phi$ as follows. If $n$ is even then $\bar \phi(n)=2$, otherwise $\bar \phi(n)=1$. Let $\phi(n) = \{2+\bar \phi(n), 4+\bar \phi(n), \dots, n\}$. For example $\phi(8)=\{4,6,8\}$ and $\phi(9)=\{3,5,7,9\}$.

Proposition \ref{gt569wp} states that 
\begin{equation}
\label{eqtn_tu7o881}
2\vert \gamma(w,n) \vert + \vert F_p(w,n-2) \vert \geq \vert F_p(w,n)\vert\mbox{.}
\end{equation} 
It follows that:
\begin{equation}
\label{eqtn_tu7o882}
2\vert \gamma(w,n-2) \vert + \vert F_p(w,n-4) \vert \geq\vert F_p(w,n-2) \vert\mbox{.}
\end{equation}
From (\ref{eqtn_tu7o881}) and (\ref{eqtn_tu7o882}):
\begin{equation}
\label{eqtn_tu7o884}
2\vert \gamma(w,n) \vert + 2\vert \gamma(w,n-2)\vert + F_p(w,n-4) \vert \geq \vert F_p(w,n)\vert\mbox{.}
\end{equation}
In general (\ref{eqtn_tu7o881}) implies that:
\begin{equation}
\label{eqtn_tu7o885}
2\vert \gamma(w,n-i) \vert + \vert F_p(w,n-2i) \vert\geq \vert F_p(w,n-i) \vert\mbox{.}
\end{equation}
Then by iterative applying of (\ref{eqtn_tu7o885}) to (\ref{eqtn_tu7o884}) we obtain that:
\begin{equation}
\label{eqtn_f589dq}
\sum_{j \in \phi(n)} 2\vert \gamma(w,j)\vert + \vert F_p(w,\bar \phi(n))\vert \geq \vert F_p(w,n)\vert
\end{equation}

We have that $\vert F_p(w,\bar \phi(n))\vert\leq q$;  just consider that $\vert F_p(w,\bar \phi(n))\vert$ is the number of palindromes of length $1$ or $2$.
Recall that $\Gamma(w,n)\geq \vert \gamma(w,j)\vert$ for $2<j<n$ and realize that $\vert\phi(n)\vert\leq \lfloor \frac{n}{2}\rfloor$. It follows from (\ref{eqtn_f589dq}) that 
$2\lfloor\frac{n}{2}\rfloor\Gamma(w,n)+q\geq \vert F_p(w,n)\vert$. Since $2\lfloor\frac{n}{2}\rfloor\leq n$ we obtain that 
$n\Gamma(w,n)+q\geq \vert F_p(w,n)\vert$.

It is easy to see that $(q+1)n\Gamma(w,n)\geq n\Gamma(w,n)+q$ for $n>0$. We prefer to use $(q+1)n\Gamma(w,n)$ instead of $n\Gamma(w,n)+q$, because this will be easier to handle in Corollary \ref{kop552q1}, even if it makes the upper bound ``a little bit worse''. This completes the proof.
\end{proof}
We need to cope with the longest proper palindromic suffixes that are ``too long''. We show that if the longest proper palindromic suffix $\lpps(v)$ is longer the half of the length of $v$, then  $v$ contains a ``short'' palindromic factor, that uniquely determines $v$. Some similar results can be found in \cite{Manea2017_Lyndon}.
\begin{lemma}
\label{kod654ih}
Let $u,v\in P^+$ be palindromes such that $u$ is a prefix of $v$ and $\frac{1}{2}\vert v\vert\leq\vert u \vert<\vert v\vert$.
Let $n,k$ be as follows:
\[
n=\begin{cases}
\lceil\frac{\vert v\vert}{2}\rceil & \vert  v\vert \mbox{ is odd}\\
\frac{\vert v\vert}{2} & \vert v\vert \mbox{ is even.}
\end{cases}
\]
\[
k=
\begin{cases}
\lceil\frac{\vert u\vert}{2}\rceil & \vert u\vert\mbox{ is odd}\\
\frac{\vert u\vert}{2}+1 & \vert u\vert\mbox{ is even.}
\end{cases}
\]
Furthermore we define $\bar \rho(u,v)=v[k,n]=v_kv_{k+1}\cdots v_{n-1}v_n$ and we define $\rho(u,v)$ as follows:
\[
\rho(u,v)=
\begin{cases}
v_nv_{n-1}\cdots v_{k+1}v_kv_kv_{k+1}\cdots v_{n-1}v_n & \vert u\vert\mbox{ is even} \\
v_nv_{n-1}\cdots v_{k+1}v_kv_{k+1}\cdots v_{n-1}v_n & \vert u\vert\mbox{ is odd.}
\end{cases}
\]
If $u_1,u_2,v_1,v_2\in P^+$,  $\vert v_1\vert=\vert v_2\vert$, $u_1$ is a prefix of $v_1$, $u_2$ is a prefix of $v_2$,  $\frac{1}{2}\vert v_1\vert\leq\vert u_1 \vert$, $\frac{1}{2}\vert v_2\vert\leq\vert u_2 \vert$, and $\rho(u_1,v_1)=\rho(u_2,v_2)$, then $v_1=v_2$. 
In other words the palindrome $\rho(u,v)$ and the length $\vert v \vert$ uniquely determine the palindrome $v$.
\end{lemma}
\begin{proof}
Given $n,j$  such that $1\leq j\leq n$, we define $\mirror(n,j)=n-j+1$.

For example: $\mirror(10,3)=8$, $\mirror(10,8)=3$, $\mirror(9,5)=5$.\\ It is easy to see that $\mirror(n,\mirror(n,j))=j$.

If $w$ is a palindrome with $\vert w\vert=t$, then clearly $w[i]=w[\mirror(t,i)]$ for $1\leq i\leq t$.

Given $u,v,k,n$ as described in the lemma. We show that for every $1\leq j< k$, there is $\bar j$ such that $j<\bar j\leq n$ and $v[j]=v[\bar j]$.
If $i=\mirror(\vert u\vert, j)$, then $v[j]=v[i]$, since the prefix $u$ is a palindrome. Clearly $i\geq k$; if $i\leq n$, then we are done: $i=\bar j$. If $i>n$, then let $\bar j=\mirror(\vert v\vert, i)$. Obviously $j<\bar j\leq n$, since $\vert u\vert <\vert v\vert$. Thus we showed that for any $1\leq j< k$ there is an index $\bar j$ which is ``closer to the center of $v$'' and such that $v[j]=[\bar j]$. Repeating the process, we can show that for any index $j$ there is an index $\tilde j$ such that $k\leq \tilde j\leq n$ and $v[j]=v[\tilde j]$. Then it is a simple exercise to reconstruct $v$ from $\rho(u,v)$ and the length $\vert v\vert$; note that from $\vert \rho(u,v)\vert $ you can determine if $\vert u\vert$ is odd or even. This proves the lemma.
\end{proof}
Let us have a look on the next examples that illuminate the proof:
\begin{example}
Let $v=12321232123212321$, $\vert v\vert=17$, $u=1232123212321$, $\vert u\vert=13$.
Then $n=9$, $k=7$, $\bar \rho(u,v)=321$, and $\rho=12321$.

Let $j=2$, then $v[2]=2$, $\mirror(\vert u\vert,2)=\mirror(13,2)=12$, $v[12]=2$, $\bar j= \mirror(\vert v\vert, 12)=\mirror(17,12)=6$, $v[6]=2$.

Let $j=6$, then $v[6]=2$, $\bar j=\mirror(\vert u\vert,6)=\mirror(13,6)=8$, $v[8]=2$.
\end{example}
\begin{example}
Let $v=211221122112$, $\vert v\vert=12$, $u=21122112$, $\vert v\vert=8$. 
Then $n=6$, $k=5$, $\bar \rho(u,v)=21$, and $\rho=1221$.
\end{example}
We derive an upper bound for the number of $u$-switches.
\begin{proposition}
\label{ujw23op}
If $w\in R$ and $n>0$ then
$$\Gamma(w,n)\leq q^5(\lceil\frac{n}{2}\rceil)^2\Gamma(w,\lceil\frac{n}{2}\rceil)\mbox{.}$$
\end{proposition} 
\begin{proof}
If a word $w$ has a palindromic factor $v$, where $v$ has a palindromic suffix $u$ with $\frac{1}{2}\vert v\vert\leq\vert u \vert<\vert v\vert$ then the longest proper palindromic suffix $\lpps(v)$ is longer than the half of $v$; formally $\vert \lpps(u)\vert\geq \frac{1}{2}\vert v\vert$. In such a case the $\rho(u,v)$ is defined and $\vert\rho(u,v)\vert\leq \frac{1}{2}\vert v\vert$. This implies that we have a ``short'' palindrome $\rho(u,v)$ or $\lpps(v)$ that uniquely identifies at most $q(q-1)$ distinct palindromes $v$ of length $n$. This means that we need only to take into account palindromic factors of length $\leq \lceil \frac{n}{2}\rceil$. Let us express this idea formally.

We partition $\trim(\gamma(w,n))$ into sets  $\Delta_{\rho}(w,n), \Delta_{\lpps}(w,n)$ as follows. Let $v\in \trim(\gamma(w,n))$. If $\frac{1}{2}\vert v\vert\leq \vert \lpps(v)\vert$ then $v\in \Delta_{\rho}(w,n)$ otherwise $v\in \Delta_{\lpps}(w,n)$. Obviously 
$\trim(\gamma(w,n)) = \Delta_{\rho}(w,n) \cup \Delta_{\lpps}(w,n)$ and $\Delta_{\rho}(w,n) \cap \Delta_{\lpps}(w,n)=\emptyset$. 
Let us investigate the size of $\Delta_{\rho}(w,n)$ and $\Delta_{\lpps}(w,n)$. 
\begin{itemize}
\item
$\rho(u,v)$ and $\vert v\vert$ uniquely determine the palindrome $v$; see Lemma \ref{kod654ih}. In addition, note that $\vert \rho(u,v)\vert \leq \lceil\frac{\vert v\vert}{2}\rceil$. Hence the number of all palindromic factors of $w$ of length $\leq \lceil\frac{n}{2}\rceil$ must be bigger or equal to the size of $\Delta_{\rho}(w,n)$:
\begin{equation}
\label{eqtn_ty540dm}
\vert  \Delta_{\rho}(w,n)\vert \leq \sum_{j=1}^{\lceil\frac{n}{2}\rceil}\vert F_p(w,j)\vert\mbox{.}
\end{equation} 
\item
$\lpps(v)$ identifies at most $q(q-1)$ distinct palindromic factors of $w$; see Corollary \ref{heu586wlp8}. It follows from the definition of $\Delta_{\lpps}(w,n)$ that  $\vert \lpps(v)\vert<\frac{1}{2}\vert v\vert$ for $v\in \Delta_{\lpps}(w,n)$. Hence the number of all palindromic factors of $w$ of length $\leq \lceil\frac{n}{2}\rceil$ multiplied by $q(q-1)$ must be bigger or equal to the size of $\Delta_{\lpps}(w,n)$:
\begin{equation}
\label{eqtn_ty541dm}
\vert \Delta_{\lpps}\vert \leq q(q-1) \sum_{j=1}^{\lceil\frac{n}{2}\rceil}\vert F_p(w,j)\vert\mbox{.}
\end{equation}
\end{itemize}
Actually the sets $\Delta_{\rho}(w,n)$ and $\Delta_{\lpps}(w,n)$ contain palindromes of length $n-2$, thus it would be sufficient to sum up to the length $\lceil\frac{n-2}{2}\rceil$ instead of $\lceil\frac{n}{2}\rceil$, but again in Corollary \ref{kop552q1} it will be more comfortable to handle $\lceil\frac{n}{2}\rceil$.

It is easy to see that $\vert\gamma(w,n)\vert\leq q(q-1)\vert \trim(\gamma(w,n)) \vert$, because for every $u\in \trim(\gamma(w,n))$ and $a,b\in A$, where $a\not =b$, there can be $aub\in \gamma(w,n)$). It follows that:
\begin{equation}
\label{eqtn_ty54dm}
\vert\gamma(w,n)\vert\leq q(q-1)\vert \trim(\gamma(w,n))\vert=q(q-1)(\vert  \Delta_{\rho}(w,n)\vert + \vert \Delta_{\lpps}(w,n)\vert)
\end{equation}
Then it follows from (\ref{eqtn_ty540dm}), (\ref{eqtn_ty541dm}) and (\ref{eqtn_ty54dm}) that 
\begin{equation}
\label{eqtn_op3411je}
\vert\gamma(w,n)\vert\leq q(q-1)(q(q-1)+1)\sum_{j=1}^{\lceil\frac{n}{2}\rceil}\vert F_p(w,j)\vert
\end{equation}
From Lemma \ref{pfw88q} we know that $\vert F_p(w,j)\vert \leq (q+1)j\Gamma(w,j) $. Therefore we have that:
\begin{equation}
\label{eqtn_hu56qpd1}
\sum_{j=1}^{\lceil\frac{n}{2}\rceil}\vert F_p(w,j)\vert\leq \sum_{j=1}^{\lceil\frac{n}{2}\rceil}(q+1)j\Gamma(w,j)\leq \lceil\frac{n}{2}\rceil (q+1)\lceil\frac{n}{2}\rceil\Gamma(w,\lceil\frac{n}{2}\rceil)
\end{equation}
From (\ref{eqtn_op3411je}) and (\ref{eqtn_hu56qpd1}):
\begin{equation}
\label{eqtn_hu57qpd1}
\vert\gamma(w,n)\vert\leq q(q-1)(q(q-1)+1)(q+1)(\lceil\frac{n}{2}\rceil)^2\Gamma(w,\lceil\frac{n}{2}\rceil)\mbox{.}
\end{equation}
To simplify (\ref{eqtn_hu57qpd1}) we apply that:
\begin{equation}
\label{eqtn_hu58qpd1}
q(q-1)(q(q-1)+1)(q+1)=q(q^2-1)(q^2-q+1)<q^5\mbox{.}
\end{equation} 
As a result  we have from (\ref{eqtn_hu57qpd1}) and (\ref{eqtn_hu58qpd1}): $\vert\gamma(w,n)\vert\leq q^5(\lceil\frac{n}{2}\rceil)^2\Gamma(w,\lceil\frac{n}{2}\rceil)$, which implies the proposition: 
$$\Gamma(w,n)=\max\{1,\max\{\vert\gamma(w,j)\vert \mid 0\leq j\leq n\}\} \leq q^5(\lceil\frac{n}{2}\rceil)^2\Gamma(w,\lceil\frac{n}{2}\rceil)\mbox{.}$$
\end{proof} 
We will need the following lemma in the proof of Corollary \ref{kop552q1}. 
\begin{lemma}
\label{lmtu85w1rt}
$\prod_{j\geq 1}^k\lceil\frac{n}{2^j}\rceil\leq (2\sqrt{n})^{\log_2{n}}$, where $k=\lfloor \log_2{n}\rfloor$.
\end{lemma}
\begin{proof}
\begin{multline*}
\prod_{j\geq 1}^k\lceil\frac{n}{2^j}\rceil=\lceil\frac{n}{2}\rceil\lceil\frac{n}{4}\rceil\lceil\frac{n}{8}\rceil \cdots \lceil\frac{n}{2^{k-1}}\rceil \lceil\frac{n}{2^{k}}\rceil\leq \\
(\frac{n}{2}+1)(\frac{n}{4}+1)(\frac{n}{8}+1)\cdots (\frac{n}{2^{k-1}}+1)(\frac{n}{2^{k}}+1)=\\
(\frac{n+2}{2})(\frac{n+4}{4})(\frac{n+8}{8})\cdots (\frac{n+2^{k-1}}{2^{k-1}})(\frac{n+2^{k}}{2^{k}})=\frac{\prod_{j=1}^k(n+2^j)}{\prod_{j=1}^k2^j}\mbox{.}
\end{multline*}
Hence we have:
\begin{equation}
\label{eqpo56e001}
\prod_{j\geq 1}^k\lceil\frac{n}{2^j}\rceil\leq\frac{\prod_{j=1}^k(n+2^j)}{\prod_{j=1}^k2^j}\mbox{.}
\end{equation}
Next we investigate the both products on the right side of (\ref{eqpo56e001}):
\begin{equation}
\label{eqpo56e002}
\prod_{j=1}^k(n+2^j)=(n+2)(n+4)(n+8)\cdots (n+2^{k-1})(n+2^k)\leq (2n)^k\mbox{.}
\end{equation}
Note that $n+2^j\leq 2n$, where $j\leq k$.
\begin{equation}
\label{eqpo56e003}
\prod_{j=1}^k2^j=22^22^3\cdots 2^{k-1}2^{k}=2^{\sum_{j=1}^kj}=2^{\frac{k(k+1)}{2}}\mbox{.}
\end{equation}

Then from (\ref{eqpo56e001}), (\ref{eqpo56e002}) and (\ref{eqpo56e003}):
$\prod_{j\geq 1}^k \lceil\frac{n}{2^j}\rceil\leq\frac{\prod_{j=1}^k(n+2^j)}{\prod_{j=1}^k2^j}\leq \frac{(2n)^k}{2^{\frac{k(k+1)}{2}}}=\left(\frac{2n}{2^{\frac{(k+1)}{2}}}\right)^k$.

Since $2^{k+1}\geq n$:

$\left(\frac{2n}{2^{\frac{(k+1)}{2}}}\right)^k\leq\left(\frac{2n}{n^{\frac{1}{2}}}\right)^k=(2n^{\frac{1}{2}})^k\leq (2\sqrt{n})^{\log_2{n}}$.
This completes the proof.
\end{proof}
Based on Proposition \ref{ujw23op} we will derive a non-recurrent upper bound for $\Gamma(w,n)$.
\begin{corollary}
\label{kop552q1}
If $w\in R$ and $n>0$ then
$$\Gamma(w,n)\leq (4q^{10}n)^{\log_2{n}}\mbox{.}$$
\end{corollary}
\begin{proof}
Proposition \ref{ujw23op} states that 
\begin{equation}
\label{eqtn_y23hj_001}
\Gamma(w,n)\leq q^5(\lceil\frac{n}{2}\rceil)^2\Gamma(w,\lceil\frac{n}{2}\rceil)\mbox{.}
\end{equation} 
The inequality (\ref{eqtn_y23hj_001}) implies that 
\begin{equation}
\label{eqtn_y23hj_002}
\Gamma(w,\lceil\frac{n}{j}\rceil)\leq q^5(\lceil\frac{n}{2j}\rceil)^2\Gamma(w,\lceil\frac{n}{2j}\rceil)\mbox{.}
\end{equation}
From (\ref{eqtn_y23hj_001}) and (\ref{eqtn_y23hj_002}):
\begin{equation}
\label{eqtn_y23hj_005}
\begin{split}
\Gamma(w,n)\leq q^5(\lceil\frac{n}{2}\rceil)^2 \Gamma(w,\lceil\frac{n}{2}\rceil) \leq q^5(\lceil\frac{n}{2}\rceil)^2q^5 (\lceil\frac{n}{4}\rceil)^2 \Gamma(w,\lceil\frac{n}{4}\rceil)\leq \\ 
q^5(\lceil\frac{n}{2}\rceil)^2q^5 (\lceil\frac{n}{4}\rceil)^2q^5 (\lceil\frac{n}{8}\rceil)^2\Gamma(w,\lceil\frac{n}{8}\rceil)\leq \dots \leq \\
\left(\prod_{j\geq 1}^{\lfloor\log_2{n}\rfloor}q^5\lceil\frac{n}{2^j}\rceil\right)^2\Gamma(w,2)\mbox{.}
\end{split}
\end{equation}
Recall that $\lceil\frac{\lceil m\rceil }{2}\rceil=\lceil\frac{ m }{2}\rceil$, where $m\geq 0$ is a real constant (see \cite{KNUTH_CM_FOUNDATION} in chapter $3.2$ Floor/ceiling applications) and  
realize that $\frac{n}{2^{\lfloor\log_2{n}\rfloor}}\geq 1$ and $\frac{n}{2^{\lceil\log_2{n}\rceil}}\leq 1$. 
Knowing that $\Gamma(w,2)=1$ and 
using Lemma \ref{lmtu85w1rt} we obtain from (\ref{eqtn_y23hj_005}): $$\Gamma(w,n)\leq \left((q^{5}2\sqrt{n})^{\log_2{n}}\right)^2\Gamma(w,2)=(4q^{10}n)^{\log_2{n}}\mbox{.}$$ This ends the proof.
\end{proof}
From Lemma \ref{pfw88q} and Corollary \ref{kop552q1} it follows easily:
\begin{corollary}
\label{kop558g4}
$\vert F_p(w,n)\vert\leq (q+1)n(4q^{10}n)^{\log_2{n}}$ where $w\in R$ and $n>0$.
\end{corollary}
We can simply apply the upper bound for the palindromic complexity to construct an upper bound for the factor complexity:
\begin{corollary}
\label{fhyu58wed}
$\vert F(w,n)\vert\leq (q+1)^2n^4(4q^{10}n)^{2\log_2{n}}$ where $w\in R$ and $n>0$.
\end{corollary}
\begin{proof}
We apply again the property of rich words that every factor is determined by its longest palindromic prefix and its longest palindromic suffix \cite{BuLuGlZa2}.
If there are at most $t$ palindromic factors in $w$ of length $\leq n$, then clearly there can be at most $t^2$ different factors of length $n$.
Let $\hat F_p(w,k)=\max\{\vert F_p(w,j)\vert \mid 0\leq j\leq k\}$. From Lemma \ref{kop558g4} we can deduce that $t\leq \sum_{i=1}^n \vert F_p(w,i)\vert \leq  n\hat F_p(w,n)\leq n(q+1)n(4q^{10}n)^{\log_2{n}}$. The lemma follows.
\end{proof}

\section{Rich words closed under reversal}
We can improve our upper bound for the factor complexity if we use the inequality (\ref{rju589w65f5n2m6a}). This inequality was shown for infinite words whose set of factors is closed under reversal. The next proposition and corollary generalize the existing proof for every word $w\in A^{+}\cup A^{\infty}$ with $F(w,n+1)$ closed under reversal.

\begin{proposition}
\label{lhey544wp}
If $w\in A^+$, $F(w,n+1)$ is closed under reversal, $\vert w\vert\geq n+1$, and $n>0$ then 
\begin{equation}\label{g6d58ww5d65h4s}\vert F_p(w,n)\vert +\vert F_p(w,n+1)\vert \leq \vert F(w,n+1)\vert-\vert F(w,n)\vert+2\mbox{.}\end{equation}
\end{proposition}
\begin{proof}
The inequality (\ref{g6d58ww5d65h4s}) is shown in \cite{BaMaPe} for infinite words closed under reversal. However, inspecting the proof of Theorem 1.2 (ii) in \cite{BaMaPe} we conclude that the inequality (\ref{g6d58ww5d65h4s}) holds if the Rauzy graph $\Gamma_n$ is strongly connected and if $L_{n+1}(w)$ (or $F(w,n+1)$ with our notation) is closed under reversal, since the map $\rho$ is then defined for all vertices and edges of the Rauzy graph $\Gamma_n$. See in \cite{BaMaPe} for the details of a construction of the Rauzy graph. Because we require $F(w,n+1)$ to be closed under reversal, we need only to prove that the Rauzy graph $\Gamma_n$ is strongly connected. 
For a finite word $w$ with $F(w,n+1)$ closed under reversal, the Rauzy graph $\Gamma_n$ is not necessarily strongly connected.  We show that the inequality (\ref{g6d58ww5d65h4s}) still holds.

Let $B$ denote the alphabet $B=A\cup\{x\}$; without loss of generality suppose that $x\not \in A$. Consider the word $\tilde w=wxw^R$ on the alphabet $B$. The word $\tilde w$ is closed under reversal,  since $\tilde w$ is a palindrome. 
Let $k\in \{n,n+1\}$ and \[\bar F(w,k)= \{uxv\mid u,v\in A^*\mbox{ and } uxv\in F(\tilde w)\mbox{ and }\vert uxv\vert=k\}\mbox{.}\]  
It is easy to see  that $F(\tilde w,k)=F(w,k)\cup \bar F(w,k)$. 
Obviously $\vert \bar F(w,k)\vert=k$. 
It follows that $F(w,n)\subset F(\tilde w,n)$, $F(w,n+1)\subset F(\tilde w,n+1)$,  
\begin{equation} 
\label{eqt58hp610}
\vert F(\tilde w,n)\vert = \vert F(w,n)\vert+n\mbox{,}
\end{equation}
and
\begin{equation}
\label{eqt58hp611}
\vert F(\tilde w,n+1)\vert = \vert F( w,n+1)\vert+n+1\mbox{.}
\end{equation}
There is just one palindrome in $\bar F(w,n)\cup \bar F(w,n+1)$, because every word in $\bar F(w,n)\cup \bar F(w,n+1)$ has exactly one occurrence of $x$. Consequently exactly one word in $\bar F(w,n)\cup \bar F(w,n+1)$ has the form $uxu^R$ for some $u\in A^*$. 
It follows that: 
\begin{equation}
\label{eqt58hp612}
\vert F_p(w,n)\vert + \vert F_p(w,n+1)\vert + 1=\vert F_p(\tilde w,n)\vert + \vert F_p(\tilde w,n+1)\vert
\end{equation}
The Rauzy graph $\tilde \Gamma_n$ of $\tilde w=wxw^R$ is strongly connected. Realize that $F(w,n+1)$ closed under reversal implies that $F(w,n+1)=F(w^R,n+1)$ and $F(w,n)=F(w^R,n)$. Hence we have that:
\begin{equation}
\label{eqt58hp6188g}
\vert F_p(\tilde w,n)\vert +\vert F_p(\tilde w,n+1)\vert \leq \vert F(\tilde w,n+1)\vert-\vert F(\tilde w,n)\vert+2\mbox{.}
\end{equation}
It follows then from (\ref{eqt58hp610}), (\ref{eqt58hp611}), (\ref{eqt58hp612}), and (\ref{eqt58hp6188g}) that: $$\vert F_p(w,n)\vert +\vert F_p(w,n+1)\vert \leq \vert F(w,n+1)\vert-\vert F(w,n)\vert+2\mbox{.}$$ This ends the proof.
\end{proof}
To clarify the previous proof, let us have a look at the example below:
\begin{example}
If $A=\{0,1\}$ and $w=1100100010011001010\in R$ then
\begin{itemize}
\item
$F(w,3)= \{110,100,001,010,000,011,101\}, \vert F(w,3)\vert=7$,
\item
$F(w,4)= \{1100,1001,0010,0100,1000,0001,0011,0110,0101,1010\},\\ \vert F(w,4)\vert=10$,
\item
$F_p(w,3)= \{010,000,101\}, \vert F_p(w,3) \vert=3$, and
\item
$F_p(w,4)=\{1001,0110\}, \vert F_p(w,4)\vert=2$.
\end{itemize}
It follows that:
\begin{itemize}
\item
$\vert F_p(w,3)\vert +\vert F_p(w,4)\vert = \vert F(w,4)\vert-\vert F(w,3)\vert+2$.
\item
$3+2=10-7+2$.
\item
$B=\{0,1,x\}$.
\item
$\tilde w=wxw^R=1100100010011001010x0101001100100010011$.
\item
$\bar F(\tilde w,3)=\{10x,0x0,x01\}$.
\item
$\bar F(\tilde w,4)=\{010x,10x0,0x01,x010\}$.
\item
$(\bar F(\tilde w,3)\cup \bar F(\tilde w,4)) \cap F_p(\tilde w)=\{0x0\}$.
\item
$\vert F_p(\tilde w,3)\vert +\vert F_p(\tilde w,4)\vert = \vert F(\tilde w,4)\vert-\vert F(\tilde w,3)\vert+2$.
\item
Thus $4+2=14-10+2$.
\end{itemize}
\end{example}
The next corollary generalizes the Proposition \ref{lhey544wp} to infinite words.
\begin{corollary}
If $w\in A^{\infty}$ with $F(w,n+1)$ closed under reversal, then 
$$\vert F_p(w,n)\vert +\vert F_p(w,n+1)\vert \leq \vert F(w,n+1)\vert-\vert F(w,n)\vert+2\mbox{.}$$
\end{corollary}
\begin{proof}
Suppose that there is $w\in A^{\infty}$ with $F(w,n+1)$ such that 
$$\vert F_p(w,n)\vert +\vert F_p(w,n+1)\vert > \vert F(w,n+1)\vert-\vert F(w,n)\vert+2\mbox{.}$$
Then obviously there is a prefix $v\in A^+$ of $w$ such that $$F(v,n+1)=F(w,n+1)$$ 
and 
$$\vert F_p(v,n)\vert +\vert F_p(v,n+1)\vert > \vert F(v,n+1)\vert-\vert F(v,n)\vert+2\mbox{.}$$ This is a contradiction Proposition \ref{lhey544wp}, since $v$ is a finite word with $F(v,n+1)$ closed under reversal.
\end{proof}
For rich words the inequality may be replaced with equality. This result has been proved also in \cite{BuLuGlZa}.
\begin{lemma}
\label{lmh_58fe6}
If $w\in R$ is a rich word such that $F(w,n+1)$ is closed under reversal, $\vert w\vert\geq n+1$, and $n>0$, then $$\vert F_p(w,n)\vert +\vert F_p(w,n+1)\vert= \vert F(w,n+1)\vert-\vert F(w,n)\vert+2\mbox{.}$$
\end{lemma}
\begin{proof}
In the proof of Proposition \ref{lhey544wp} the word $\tilde w=wxw^R$ is rich if $w$ is rich. To see this, suppose that $w$ is rich. Then $wx$ is rich, because $\lps(wx)=x$, which is a unioccurrent palindrome in $wx$ and $wxw^R$ is a palindromic closure of $wx$, which preserves richness \cite{GlJuWiZa}. 
Then the equality
follows from Proposition 3 in \cite{BrRe-conjecture}; the proposition uses the
palindromic defect $D(w)$ of a word, which is, by definition, equal to
zero for a rich word.
\end{proof}
Based on Lemma \ref{lmh_58fe6} we can present a new relation for palindromic and factor complexity.
\begin{proposition}
\label{uq222kp69}
Let $\hat F_p(w,k)=\max\{\vert F_p(w,j)\vert \mid 0\leq j\leq k\}$.
If $w\in R$ is a rich word such that $F(w,n+1)$ is closed under reversal, $\vert w\vert\geq n+1$, and $n>0$, then
$$\vert F(w,n)\vert\leq 2(n-1)\hat F_p(w,n)-2(n-1) +q\mbox{.}$$
\end{proposition}
\begin{proof}
Lemma \ref{lmh_58fe6} states that: 
\begin{equation}
\label{eqtn_revz45_898}
\vert F_p(w,n)\vert +\vert F_p(w,n+1)\vert-2 = \vert F(w,n+1)\vert-\vert F(w,n)\vert\mbox{.}
\end{equation}
Since $F(w,n+1)$ closed under reversal, we have that $F(w,i)$ is closed under reversal for $i\leq n+1$. We can sum (\ref{eqtn_revz45_898}) over all lengths $i\leq n$:
\begin{equation}
\label{eqtn_revz45_1}
\sum_{i=1}^{n-1} (\vert F_p(w,i)\vert +\vert F_p(w,i+1)\vert-2)=\sum_{i=1}^{n-1} (\vert F(w,i+1)\vert-\vert F(w,i)\vert)\mbox{.}
\end{equation}
The sums from (\ref{eqtn_revz45_1}) may be expressed as follows:
\begin{equation}
\label{eqtn_revz45_2}
\begin{split}
\sum_{i=1}^{n-1} (\vert F(w,i+1)\vert-\vert F(w,i)\vert) = F(w,2)-F(w,1) + F(w,3)-F(w,2) \\ + F(w,4)-F(w,3)+\dots + F(w,n-1)-F(w,n-2) \\ +F(w,n)-F(w,n-1) = F(w,n)-F(w,1)\mbox{.}
\end{split}
\end{equation}
\begin{equation}
\label{eqtn_revz45_3}
\sum_{i=1}^{n-1} (\vert F_p(w,i)\vert +\vert F_p(w,i+1)\vert-2)\leq (n-1)(\hat F_p(w,n-1) + \hat F_p(w,n)-2)\mbox{.}
\end{equation}
From (\ref{eqtn_revz45_1}), (\ref{eqtn_revz45_2}), and (\ref{eqtn_revz45_3}) we get:
$$F(w,n)-F(w,1)\leq (n-1)(\hat F_p(w,n-1) + \hat F_p(w,n)-2)\mbox{.}$$
It follows that:
$$F(w,n)\leq (n-1)(2\hat F_p(w,n)-2) +F(w,1)\mbox{.}$$
This can be reformulated as:
$$F(w,n)\leq 2(n-1)\hat F_p(w,n)-2(n-1) +F(w,1)\mbox{.}$$
Since $F(w,1)\leq q$ it follows that:
$$F(w,n)\leq 2(n-1)\hat F_p(w,n)-2(n-1) +q\mbox{.}$$
This completes the proof.
\end{proof}
Proposition \ref{uq222kp69} and Lemma \ref{kop558g4} imply an improvement to our upper bound for the factor complexity for rich words with $F(w,n+1)$ closed under reversal:
\begin{corollary}
\label{pr_65dfr69we85}
If $w\in R$ with $F(w,n+1)$ closed under reversal, $\vert w\vert\geq n+1$, and $n>0$, then:
$$\vert F(w,n)\vert \leq 2(n-1)(q+1)n(4q^{10}n)^{\log_2{n}}-2(n-1) +q\mbox{.}$$
\end{corollary}
Since the palindromic closure of finite rich words is closed under reversal, we can improve the upper bound for factor complexity for finite rich words.
\begin{corollary}
If $w\in R$ and $n>0$ then
$$\vert F(w,n)\vert \leq 2(2n-1)(q+1)2n(8q^{10}n)^{\log_2{2n}}-2(2n-1) +q\mbox{.}$$
\end{corollary}
\begin{proof}
Palindromic closure $\hat w$ of a word $w\in R$ preserves richness. Furthermore $F(\hat w)$ is closed under reversal, $F(w)\subseteq F(\hat w)$, and $\vert \tilde w\vert\leq 2\vert w\vert$ \cite{GlJuWiZa}. Hence we can apply Corollary \ref{pr_65dfr69we85}, where we replace $n$ with $2n$. 
\end{proof}
The purpose of the next obvious corollary is to present a ``simple'' (although a somewhat worse) upper bound for the factor complexity.
\begin{corollary}
If $w\in R$ and $n>0$ then
$$\vert F(w,n)\vert \leq (q+1)8n^2(8q^{10}n)^{\log_2{2n}}+q\mbox{.}$$
\end{corollary}

\section*{Acknowledgments}
The author wishes to thank to Štěpán Starosta for his useful comments. The author acknowledges support by the Czech Science
Foundation grant GA\v CR 13-03538S and by the Grant Agency of the Czech Technical University in Prague, grant No. SGS14/205/OHK4/3T/14.

\bibliographystyle{siam}
\IfFileExists{biblio.bib}{\bibliography{biblio}}{\bibliography{../!bibliography/biblio}}

\end{document}